\documentclass[a4paper,english,fontsize=10pt,parskip=half,abstract=true]{scrartcl}
\usepackage{babel}
\usepackage[utf8]{inputenc}
\usepackage[T1]{fontenc}
\usepackage[a4paper,left=20mm,right=20mm,top=30mm,bottom=30mm]{geometry}
\usepackage{amsmath}
\usepackage{amsthm}
\usepackage{amssymb}
\usepackage{enumerate}
\usepackage{aliascnt}
\usepackage{multirow}
\usepackage[bookmarks=true,
            pdftitle={Groups with few p'-characters degrees in the principal block},
            pdfauthor={Benjamin Sambale},
            pdfkeywords={},
            pdfstartview={FitH}]{hyperref}

\newtheorem*{ThmA}{Theorem~A}

\newtheorem{Thm}{Theorem}[section]
\newaliascnt{Lem}{Thm}
\newtheorem{Lem}[Lem]{Lemma}
\aliascntresetthe{Lem}
\newaliascnt{Prop}{Thm}
\newtheorem{Prop}[Prop]{Proposition}
\aliascntresetthe{Prop}
\newaliascnt{Cor}{Thm}

\aliascntresetthe{Cor}
\newaliascnt{Con}{Thm}

\aliascntresetthe{Con}
\theoremstyle{definition}
\newaliascnt{Def}{Thm}
\newtheorem{Def}[Def]{Definition}
\aliascntresetthe{Def}
\newtheorem{Notation}{Notation}

\numberwithin{equation}{section}

\allowdisplaybreaks[1]

\let\ph=\phi
\renewcommand{\phi}{\varphi}
\newcommand{\C}{\operatorname{C}}
\newcommand{\N}{\operatorname{N}}
\newcommand{\Z}{\operatorname{Z}}

\newcommand{\cH}{\mathcal{H}}
\newcommand{\Aut}{\operatorname{Aut}}

\newcommand{\pcore}{\operatorname{O}}

\newcommand{\Irr}{\operatorname{Irr}}
\newcommand{\F}{\operatorname{F}}

\newcommand{\wt}[1]{\widetilde{#1}}

\newcommand{\bg}[1]{\textbf{#1}}

\newcommand{\tw}[1]{{}^#1\!}
\newcommand{\wh}[1]{\widehat{#1}}
\def\irr#1{{\Irr}(#1)}
\def\irrp#1{{\Irr}_{p'}(#1)}

\mathchardef\ordinarycolon\mathcode`\:  %defines a nice ":=" 
 \mathcode`\:=\string"8000
 \begingroup \catcode`\:=\active
   \gdef:{\mathrel{\mathop\ordinarycolon}}
 \endgroup

\title{Groups with few $p'$-character degrees\\ in the principal block}
\author{Eugenio Giannelli\footnote{Dipartimento di Matematica e Informatica, U. Dini, Viale Morgagni 67/a, Firenze, Italy,  \href{mailto:eugenio.giannelli@unifi.it}{eugenio.giannelli@unifi.it}}, Noelia Rizo\footnote{Dipartimento di Matematica e Informatica, U. Dini, Viale Morgagni 67/a, Firenze, Italy,  \href{mailto:eugenio.giannelli@unifi.it}{noelia.rizocarrion@unifi.it}},\\ Benjamin Sambale\footnote{Institut für Algebra, Zahlentheorie und Diskrete Mathematik, Leibniz Universität Hannover, Germany, \href{mailto:benjamin.sambale@uni-jena.de}{sambale@math.uni-hannover.de}} \ and A. A. Schaeffer Fry\footnote{Department of Mathematical and Computer Sciences, Metropolitan State University of Denver, Denver, CO 80217, USA, \href{mailto:aschaef6@msudenver.edu}{aschaef6@msudenver.edu}}}
\date{\today}

\begin{document}
\frenchspacing
\maketitle
\begin{abstract}\noindent
Let $p\ge 5$ be a prime and let $G$ be a finite group.
We prove that $G$ is $p$-solvable of $p$-length at most $2$ if there are at most two distinct $p'$-character degrees in the principal $p$-block of $G$. This generalizes a theorem of Isaacs--Smith as well as a recent result of three of the present authors.
\end{abstract}


\section{Introduction}

Let $G$ be a finite group. If all non-linear irreducible characters of $G$ have degree divisible by a prime $p$, then $G$ has a normal $p$-complement by a theorem of Thompson~\cite[Theorem~1]{Thompson} (see also \cite[Corollary~12.2]{Isaacs}). 
Moreover, Berkovich~\cite[Proposition~9 and the subsequent remark]{Berkovichsolv} has shown that $G$ is solvable in this situation.
This result was extended in Kazarin--Berkovich~\cite{KazarinBerkovich} to the case where $G$ has at most one non-linear character of $p'$-degree.
In a recent paper~\cite{GRS}, three of the present authors proved more generally that $G$ is solvable of $p$-length at most $2$ whenever $p\ge 5$ and $|\{\chi(1):\chi\in\Irr_{p'}(G)\}|\le 2$ where $\Irr_{p'}(G)$ is the set of irreducible characters of $G$ of $p'$-degree. This has solved Problem~1 in \cite[p. 588]{KazarinBerkovich} and Problem~5.3 in \cite{NavarroVar}.

In the present paper we generalize our theorem to blocks. This is motivated by a result of Isaacs and Smith~\cite[Corollary~3]{IsaacsSmith} who showed that $G$ has a normal $p$-complement if and only if all non-linear characters in the principal $p$-block of $G$ have degree divisible by $p$. The following is our main theorem.

\begin{ThmA}
Let $B_0$ be the principal block of a finite group $G$ with respect to a prime $p\ge 5$. Suppose that $|\{\chi(1):\chi\in\Irr_{p'}(B_0)\}|\le 2$. Then $G/\pcore_{p'}(G)$ is solvable and $\pcore^{pp'pp'}(G)=1$. In particular, $G$ is $p$-solvable.
\end{ThmA}

As usual we define $\pcore^{pp'}(G):=\pcore^{p'}(\pcore^{p}(G))$ and so on. It is easy to construct groups of $p$-length $2$ satisfying the hypothesis of Theorem~A (e.\,g. $G=(C_5^5\rtimes C_{11})\rtimes C_5$ with $p=5$). 
In contrast to the main theorem of \cite{GRS} we cannot conclude further that $G$ is solvable since every $p'$-group satisfies the assumption of Theorem~A. Furthermore, the examples given in \cite{GRS} show that Theorem~A does not extend to $p\in\{2,3\}$. 
We also like to mention a conjecture by Malle and Navarro~\cite{MNchardeg}, which generalizes the result of Isaacs and Smith to arbitrary blocks. More precisely, they conjectured that a $p$-block $B$ of $G$ is nilpotent if and only if all height $0$ characters in $B$ have the same degree. We do not know if our main result admits a version for arbitrary blocks. 

%As in \cite{GRS}, 
The proof of Theorem~A relies on the classification of finite simple groups.
In the next section we reduce Theorem~A to a statement about simple groups (\autoref{simple} below), which is proved case-by-case in the following two sections. 
We care to remark that in the case of alternating groups, \autoref{simple} is deduced as a consequence of a more general statement giving a lower bound for the number of (extendable) $p'$-character degrees in \emph{any} block of maximal defect. This is \autoref{cor: an} below, which we believe is of independent interest. 

\section{Reduction to simple groups}

The following proposition about simple groups will be proven in the next two sections.

\begin{Prop}\label{simple}
Let $S$ be a finite non-abelian simple group of order divisible by a prime $p\ge 5$. 
\begin{enumerate}[(i)]
\item If $S\ne P\Omega_8^+(q)$, then there exist $\alpha,\beta\in\Irr(S)$ with the following properties:
\begin{itemize}
\item $\alpha\ne 1\ne\beta$,
\item $\alpha(1)$ and $\beta(1)$ are not divisible by $p$,
\item for every $S\le T\le\Aut(S)$, $\alpha$ extends to a character in the principal block of $T$,
\item $\beta$ lies in the principal block of $S$ and is $P$-invariant for some Sylow $p$-subgroup $P$ of $\Aut(S)$,
\item $\beta(1)\nmid\alpha(1)$.
\end{itemize}

\item If $S=P\Omega_8^+(q)$, then there exist $\alpha,\beta\in\Irr(S)$ with the following properties:
\begin{itemize}
\item $\alpha\ne 1\ne\beta$,
\item $\alpha(1)$ and $\beta(1)$ are not divisible by $p$,
\item $\alpha(1)>2\beta(1)$,
\item for every $S\le T\le\Aut(S)$ there exist $\hat\alpha,\hat\beta\in\irr{\Aut(T)}$ in the principal block such that $\hat\alpha_S\in\{\alpha,2\alpha\}$ and $\hat\beta_S\in\{\beta,2\beta\}$.
\end{itemize}
\end{enumerate}
\end{Prop}

We make use of the following results.

\begin{Lem}[Murai~{\cite[Lemma~4.3]{MuraiHZ}}]\label{Murai}
Let $N\unlhd G$ be finite groups with principal $p$-blocks $B_N$ and $B_G$ respectively. Suppose that $\psi\in\Irr_{p'}(B_N)$ is invariant under a Sylow $p$-subgroup of $G$. Then there exists a character $\chi\in\Irr_{p'}(B_G)$ lying over $\psi$.
\end{Lem}

\begin{Lem}\label{prinprod}
Let $\chi,\psi\in\Irr(B_0)$ where $B_0$ is the principal $p$-block of $G$. Suppose that $p\nmid\chi(1)$ and $\chi\psi\in\Irr(G)$. Then $\chi\psi\in\Irr(B_0)$.
\end{Lem}
\begin{proof}
Clearly, $\overline{\psi}\in\Irr(B_0)$. Hence by \cite[Corollary~3.25]{Navarro}, we have
\[[\chi\psi,1]^0=[\chi,\overline{\psi}]^0\ne 0.\]
The claim follows from \cite[Theorem~3.19]{Navarro}.
\end{proof}

Now we are in a position to reduce Theorem~A to simple groups.

\begin{Thm}\label{2degs}
If \autoref{simple} holds, then Theorem~A holds. 
\end{Thm}
\begin{proof}
Let $p$, $G$ and $B_0$ be as in Theorem~A.
Suppose first that $G$ is $p$-solvable. Let $N:=\pcore_{p'}(G)$. Then, by \cite[Theorem~10.20]{Navarro}, $\Irr(B_0)=\Irr(G/N)$. It follows from \cite[Theorem~A]{GRS} that $G/N$ is solvable and $\pcore^{pp'pp'}(G/N)=1$. In particular, $\pcore^{pp'p}(G)N/N$ is a $p'$-group. Since $N$ is a $p'$-group, this implies $\pcore^{pp'pp'}(G)=1$.

Thus, it suffices to show that $G$ is $p$-solvable. 
Let $N$ be a minimal normal subgroup of $G$. Since the principal block of $G/N$ lies in $B_0$, we may assume that $G/N$ is $p$-solvable by induction on $|G|$. 
If $N$ is a $p$-group or a $p'$-group, then we are done. Therefore, by way of contradiction, we assume that 
\[N=S_1\times\ldots\times S_t\] 
with isomorphic non-abelian simple groups $S:=S_1\cong\ldots\cong S_t$ of order divisible by $p$. Since $N$ is the unique minimal normal subgroup, $\C_G(N)=1$. Moreover, $G$ permutes $S_1,\ldots,S_t$ transitively by conjugation.

\textbf{Case~1:} $S\ne P\Omega_8^+(q)$.\\
Let $\alpha,\beta\in\Irr(S)$ as in \autoref{simple}. 
We may regard $\alpha$ as a character of $S\C_G(S)$, since $S\C_G(S)/\C_G(S)\cong S/\Z(S)=S$. As such it extends to a character $\hat\alpha$ in the principal block of $\N_G(S)$, because $\N_G(S)/\C_G(S)\le\Aut(S)$. 
Let $M:=\N_G(S_1)\cap\ldots\cap\N_G(S_t)\unlhd G$. Since the principal block of $\N_G(S)$ covers the principal block $B_M$ of $M$, the restriction $\hat\alpha_M$ lies in $B_M$. 
Now by \cite[Corollary~10.5]{Navarro2}, the tensor induction $\psi:=\hat\alpha^{\otimes G}$ is an irreducible character of $G$ with $p'$-degree $\psi(1)=\alpha(1)^t$. 
Let $x_1,\ldots,x_t\in G$ be representatives for the right cosets of $\N_G(S)$ in $G$ such that $S_1^{x_i}=S_i$. 
Then for $g\in M$ we obtain
\[
\psi(g)=\prod_{i=1}^t\hat\alpha^{x_i}(g)
\]
from \cite[Lemma~10.4]{Navarro2}. In particular, $\psi_N=\alpha^{x_1}\times\ldots\times\alpha^{x_t}\in\Irr(N)$ and therefore $\psi_M\in\Irr(M)$ as well.
Since $\hat\alpha_M$ lies in $B_M$, so does $\hat\alpha_M^{x_i}$. Hence, by \autoref{prinprod} also $\psi_M=\hat\alpha_M^{x_1}\ldots\hat\alpha_M^{x_t}$ lies in $B_M$. 

Let $Q$ be a Sylow $p$-subgroup of $M$. Then $Q\cap S_i$ is a Sylow $p$-subgroup of $S_i$. It follows that $\C_G(Q)\subseteq\C_G(Q\cap S_i)\subseteq\N_G(S_i)$ for $i=1,\ldots,t$ and therefore $\C_G(Q)\subseteq M$. Hence, the Brauer correspondent $B_M^G$ is defined (see \cite[Theorem~4.14]{Navarro}) and equals $B_0$ by Brauer's third main theorem.
Every block $B$ of $G$ covering $B_M$ has a defect group containing $Q$ by \cite[Theorem~9.26]{Navarro}.
Hence by \cite[Lemma~9.20]{Navarro}, $B$ is regular with respect to $N$ and therefore $B=B_0$ by \cite[Theorem~9.19]{Navarro}.
Thus, $B_0$ is the only block of $G$ covering $B_M$.
This implies $\psi\in\Irr_{p'}(B_0)$.
Since the trivial character in $B_0$ has degree $1$, $d:=\psi(1)=\alpha(1)^t$ is the unique non-trivial $p'$-character degree in $B_0$ by hypothesis.

Now we work with $\beta$. 
Let $P$ be a Sylow $p$-subgroup of $G$ such that $\beta$ is invariant under $\N_P(S)$. 
Without loss of generality, let $\{S_1,\ldots,S_r\}$ be a $P$-orbit. Let $y_i\in P$ such that $S_1^{y_i}=S_i$ for $i=1,\ldots,r$. Then $\beta_i:=\beta^{y_i}$ lies in the principal block of $S_i$. By \autoref{prinprod}, $\beta_1\times\ldots\times\beta_r$ lies in the principal block of $N$. Moreover, if $\beta_i^y\in\Irr(S_j)$ for some $y\in P$, then $y_iyy_j^{-1}\in\N_P(S)$. Since $\beta$ is $\N_P(S)$-invariant, it follows that $\beta_i^y=\beta^{y_iyy_j^{-1}y_j}=\beta^{y_j}=\beta_j$. This shows that $\{\beta_1,\ldots,\beta_r\}$ is $P$-orbit and $\beta_1\times\ldots\times\beta_r$ is $P$-invariant. If $r<t$, then we consider $\beta_{r+1}:=\beta^{x_{r+1}}\in\Irr(S_{r+1})$. By Sylow's theorem, we can assume after conjugation inside $\N_G(S_{r+1})$ that $\beta_{r+1}$ is $\N_P(S_{r+1})$-invariant. Now we can form the $P$-orbit of $\beta_{r+1}$ to obtain another $P$-invariant character $\beta_{r+1}\times\ldots\times\beta_s\in\Irr(N)$ in the principal block of $N$. We repeat this with every $P$-orbit and eventually get a $PN$-invariant character
\[\tau:=\beta_1\times\ldots\times\beta_t\in\Irr(N)\]
in the principal block of $N$. Since $o(\tau)=1$ and $\gcd(\tau(1),|PN/N|)=1$, $\tau$ extends to $PN$ (see \cite[Corollary~8.16]{Isaacs}). By \autoref{Murai}, there exists some $\chi\in\Irr_{p'}(B_0)$ such that $\tau$ is a constituent of $\chi_N$.
Since $1\ne \beta(1)^t=\tau(1)\mid\chi(1)$, it follows that $\chi(1)=d=\psi(1)$. But then $\beta(1)^t\mid\psi(1)=\alpha(1)^t$ and $\beta(1)\mid\alpha(1)$, a contradiction to the choice of $\alpha$ and $\beta$. 

\textbf{Case~2:} $S=P\Omega_8^+(q)$.\\
Let $\alpha,\beta\in\Irr(S)$ and $\hat\alpha,\hat\beta\in\Irr(\N_G(S))$ as in \autoref{simple}. Since the principal block of $\N_G(S)$ covers $B_M$, $\hat\alpha_M$ is the sum of at most two irreducible characters in $B_M$. If $\alpha_1\in\Irr(B_M)$ is one of those summands, then $\alpha_1^{x_1}\ldots\alpha_1^{x_t}$ restricts to $\alpha^{x_1}\times\ldots\times\alpha^{x_t}\in\Irr(N)$. Hence, by \autoref{prinprod}, $\alpha_1^{x_1}\ldots\alpha_1^{x_t}$ lies in $B_M$. As in Case~1 we see that $(\hat\alpha^{\otimes G})_M$ is a sum of irreducible characters in $B_M$. 
Moreover, $(\hat\alpha^{\otimes G})_N=d(\alpha^{x_1}\times\ldots\times\alpha^{x_t})$ where $d\in\{1,2^t\}$.
Since $B_0$ is the only block of $G$ covering $B_M$, all irreducible constituents of $\hat\alpha^{\otimes G}$ lie in $B_0$. We may choose such a constituent $\chi$ of $p'$-degree. Then $\chi_N=e(\alpha^{x_1}\times\ldots\times\alpha^{x_t})$ for some integer $e\le d\le 2^t$. Similarly, we choose a constituent $\psi$ of $\hat{\beta}^{\otimes G}$ with $p'$-degree. Then by \autoref{simple} we derive the contradiction
\[\alpha(1)^t>2^t\beta(1)^t\ge\psi(1)=\chi(1)\ge\alpha(1)^t.\qedhere\]
\end{proof}

\section{Alternating groups}\label{alt}

This section is devoted to proving \autoref{simple} for the alternating groups $S=\mathfrak{A}_n$ where $n\ge 5$. It is well-known that $\Aut(S)\cong\mathfrak{S}_n$ is the symmetric group unless $n=6$. 

Given $n\in\mathbb{N}$ we let $\mathcal{P}(n)$ be the set of partitions of $n$. Let $\lambda=(\lambda_1,\ldots, \lambda_k)\in\mathcal{P}(n)$. Adopting the notation of \cite[Chapter 1]{OlssonCombi} we let $\ell(\lambda)=k$ denote the number of parts of $\lambda$, and $\mathcal{Y}(\lambda)$ be the Young diagram of $\lambda$. Given a node $(i,j)\in\mathcal{Y}(\lambda)$ we denote by 
 $h_{ij}(\lambda)$ the length of the hook corresponding to $(i,j)$.  
If $q\in\mathbb{N}$ then the $q$-core $C_q(\lambda)$ of $\lambda$ is the partition obtained from $\lambda$ by successively removing all hooks of length $q$ (usually called $q$-hooks). We denote by $\mathcal{H}^q(\lambda)$ the subset of nodes of $\mathcal{Y}(\lambda)$ having associated hook-length divisible by $q$. 
A partition $\gamma$ is called a $q$-core if $\mathcal{H}^q(\lambda)=\emptyset$.

The set $\mathrm{Irr}(\mathfrak{S}_n)$ is naturally in bijection with $\mathcal{P}(n)$. Given $\lambda\in\mathcal{P}(n)$ we let $\chi^\lambda$ be the corresponding irreducible character of 
$\mathfrak{S}_n$. Let $p$ be a prime and $\lambda, \mu\in\mathcal{P}(n)$. By \cite[6.1.21]{JamesKerber} we know that $\chi^\lambda$ and $\chi^\mu$ lie in 
the same $p$-block of $\mathfrak{S}_n$ if and only if $C_p(\lambda)=C_p(\mu)$. If $\gamma$ is a $p$-core partition then 
we denote by $B(\mathfrak{S}_n, \gamma)$ the corresponding $p$-block of $\mathfrak{S}_n$. We use the notation $\lambda\vdash_{p'}n$ to say that $\chi^\lambda$ has degree coprime to $p$.

The following result follows from \cite{Mac} and it will be extremely useful for our purposes. 

\begin{Lem}   \label{lem: last layer}
Let $p$ be a prime and let $n$ be a natural number with $p$-adic expansion
 $n=\sum_{j=0}^ka_jp^j$. Let $\lambda$ be a partition of $n$. Then $\lambda\vdash_{p'}n$ if and only if $|\cH^{p^k}(\lambda)|=a_k$ and $C_{p^k}(\lambda)\vdash_{p'}n-a_kp^k$. 
\end{Lem}

A straightforward consequence of \autoref{lem: last layer} is that $\mathrm{Irr}_{p'}(B(\mathfrak{S}_n, \gamma))\neq\emptyset$ if and only if $|\gamma|<p$.

For $\lambda\in\mathcal{P}(n)$, we denote by $\lambda'$ its conjugate partition. 
From \cite[2.5.7]{JamesKerber} we know that $\psi^\lambda:=(\chi^\lambda)_{\mathfrak{A}_n}$ is irreducible if and only if $\lambda\neq\lambda'$. In this case $\chi^\lambda$ and $\chi^{\lambda'}$ are all the extensions of $\psi^\lambda$ to $\mathfrak{S}_n$. Let $\lambda,\mu$ be non-self-conjugate partitions of $n$. Then $\psi^\lambda$ and $\psi^\mu$ lie in the same $p$-block of $\mathfrak{A}_n$ if and only if $C_p(\lambda)\in\{C_p(\mu), C_p(\mu)'\}$. It follows that also $p$-blocks of $\mathfrak{A}_n$ can be labeled by $p$-core partitions, by keeping in mind that conjugated $p$-cores label the same $p$-block. We denote by $B(n; \gamma)$ the $p$-block of $\mathfrak{A}_n$ labeled by $\gamma$. 

In order to show that \autoref{simple} holds for alternating groups, we introduce the following conventions. 
\begin{Notation}\label{not: an}
Let $B$ be a $p$-block of $\mathfrak{A}_n$. We let $\mathrm{cd}_{p'}^{ext}(B)$ be the set of degrees of irreducible characters in $B$ of degree coprime to $p$ that extend to an irreducible character of $\mathfrak{S}_n$.
Moreover, when $S$ is a subset of $\mathcal{P}(n)$ we let $\mathrm{cd}(S)=\{\chi^\lambda(1)\ |\ \lambda\in S\}.$
\end{Notation}

Observe that if $B$ is the principal $p$-block of $\mathfrak{A_n}$ and $\psi^\lambda$ lies in $B$ and extends to $\mathfrak{S}_n$, then one of the two extensions of $\psi^\lambda$ lies in the principal $p$-block of $\mathfrak{S}_n$. This is explained in \cite{OlssonBlocks}. 
Even if in this article we are mainly interested in studying the principal block, in \autoref{cor: an} below we are going to compute an explicit lower bound for $|\mathrm{cd}_{p'}^{ext}(B(n,\gamma))|$, for any $p$-core $\gamma$ such that $|\gamma|<p$. 

Given $\gamma=(\gamma_1,\ldots, \gamma_\ell)\vdash n$ and natural numbers $x$ and $y$, we denote by $\gamma\star (x, y)$ the partition of $n+x+y$ defined by 
\[\gamma\star (x, y)=(\gamma_1+x,\gamma_2,\ldots, \gamma_\ell,1^y).\]

We start by proving a technical lemma that will be useful later in this section. 

\begin{Lem}\label{lem:alt1}
Let $p$ be a prime, let $m,n,w\in\mathbb{N}$ be such that $m<p$ and $n=m+pw$. Let $\gamma\vdash m$ and let $a\in\mathbb{N}$ be such that $\lfloor\frac{w+1}{2}\rfloor+1\leq a\leq w$. 
Setting $\lambda=\gamma\star (ap, (w-a)p)$ and $\mu=\gamma\star ((a-1)p, (w-a+1)p)$, we have that 
$\chi^\lambda(1)<\chi^\mu(1)$. 
\end{Lem}
\begin{proof}
For $\nu\vdash n$ we let $\pi(\nu):=\prod h_{ij}(\nu)$ be the product of the hook-lengths in $\nu$. From the hook length formula \cite[2.3.21]{JamesKerber} it follows that $\chi^\nu(1)\pi(\nu)=n!$. We let $h^i=h_{1i}(\gamma)$ and $h_j=h_{j1}(\gamma)$ for all $i\in\{1,\dots, \gamma_1\}$ and all $j\in\{1,\dots, \ell(\gamma)\}$.
It follows that \[\pi(\lambda)=(ap)!\cdot((w-a)p)!\cdot\prod_{i=2}^{\gamma_1}(h^i+ap)\cdot\prod_{i=2}^{\ell(\gamma)}(h_i+(w-a)p)\cdot \widehat{\gamma}\cdot (h_{11}(\gamma)+pw),\]
where $\widehat{\gamma}$ is the product of the hook lengths $h_{ij}(\gamma)$ for all $i,j\geq 2$. Similarly 
\[\pi(\mu)=((a-1)p)!\cdot((w-a+1)p)!\cdot\prod_{i=2}^{\gamma_1}(h^i+(a-1)p)\cdot\prod_{i=2}^{\ell(\gamma)}(h_i+(w-a+1)p)\cdot \widehat{\gamma}\cdot (h_{11}(\gamma)+pw).\]
It follows that $\pi(\lambda)/\pi(\mu)=A\cdot B\cdot C$, where 
\[A=\prod_{i=1}^{p}\frac{(a-1)p+i}{(w-a)p+i},\ \ B=\prod_{i=2}^{\gamma_1}\frac{h^i+ap}{h^i+(a-1)p},\ \ \text{and}\ \ 
C=\prod_{i=2}^{\ell(\gamma)}\frac{h_i+(w-a)p}{h_i+(w-a+1)p}.\]
We remark that we always regard empty products as equal to $1$.
We observe that $B\geq 1$. Since $a-1\geq w-a+1$ by hypothesis, it is clear that $A>1$. Hence, if $\ell(\gamma)=1$ then $C=1$ and clearly $A\cdot B\cdot C>1$. Suppose that $\ell(\gamma)\geq 2$. 
Then observe that $p>|\gamma|>h_2>h_3>\cdots>h_{\ell(\gamma)}\geq 1$. Hence for all $i\in\{2,\ldots, \ell(\gamma)\}$ we have that 
$\frac{(a-1)p+h_i}{(w-a)p+h_i}$ is one of the factors appearing in $A$. Moreover 
\[\frac{(a-1)p+h_i}{(w-a)p+h_i}\cdot \frac{h_i+(w-a)p}{h_i+(w-a+1)p}\geq 1,\] 
since $a-1\geq w-a+1$. 
We conclude that $A\cdot B\cdot C\geq A\cdot C>1$ and therefore that $\chi^\lambda(1)<\chi^\mu(1)$. 
\end{proof}

\begin{Def}\label{def: H-Omega}
Let $p$ be a prime and $n=wp+m$, for some $m<p$. Let  $\gamma$ be a $p$-core partition of $m$.
We let $H(n; \gamma)$ be the subset of $\mathcal{P}(n)$ defined by 
\[H(n;\gamma)=\{\lambda\vdash_{p'}n\ |\ \mathrm{C}_p(\lambda)=\gamma, \lambda=\gamma\star(a,n-m-a)\}.\]
We also set $\Omega(n; \gamma)=\{\lambda\in H(n;\gamma)\ |\ \lambda_1>(\lambda')_1\}$.
\end{Def}

\begin{Lem}\label{lem:alt3}
Let $n=\sum_{i=0}^ka_ip^{i}$ be the $p$-adic expansion of $n$, with $a_k\neq 0$. If $\gamma\vdash a_0$, then \[|\mathrm{cd}(\Omega(n; \gamma))|=|\Omega(n; \gamma)|\geq \lfloor \frac{a_{k}+1}{2}\rfloor\cdot\prod_{i=1}^{k-1}(a_i+1).\]
\end{Lem}
\begin{proof}
Let $\lambda=\gamma\star(x,n-a_0-x)$, for some $0\leq x\leq n-a_0$. Let $x=\sum_{i=0}^tb_ip^i$ be the $p$-adic expansion of $x$. 
By definition of $H(n; \gamma)$ we have that $\lambda\in H(n; \gamma)$ if and only if $\lambda\vdash_{p'}n$ and $C_p(\lambda)=\gamma$. In turn, this is equivalent to have that $p$ divides $x$ (and $n-a_0-x$) so that $C_p(\lambda)=\gamma$ and 
by \autoref{lem: last layer} to have that 
$b_0=0$ and $0\leq b_i\leq a_i$ for all $i\geq 1$. It follows that $|\mathcal{H}(n; \gamma)|=\prod_{i=1}^k(a_i+1)$. Moreover, if $b_k\geq \lfloor a_k/2\rfloor+1$, then certainly $\lambda_1>(\lambda')_1$ and therefore $\lambda\in\Omega(n;\gamma)$. It follows that 
\[|\Omega(n; \gamma)|\geq \lfloor \frac{a_{k}+1}{2}\rfloor\cdot\prod_{i=1}^{k-1}(a_i+1).\]
We conclude by observing that \autoref{lem:alt1} implies that given $\lambda,\mu\in\Omega(n;\gamma)$ we have that $\chi^\lambda(1)\neq \chi^\mu(1)$ and hence that $|\mathrm{cd}(\Omega(n; \gamma))|=|\Omega(n; \gamma)|$.
\end{proof}

Given $\lambda\in\Omega(n;\gamma)$ we have that $\chi^\lambda$ lies in $B(\mathfrak{S}_n; \gamma)$ and that $(\chi^\lambda)_{\mathfrak{A}_n}$ is irreducible and lies in $B(n; \gamma)$. As explained in Notation \ref{not: an} above, $\mathrm{cd}_{p'}^{ext}(B(n; \gamma))$ denotes the set of degrees of irreducible characters of $B(n; \gamma)$ of degree coprime to $p$ that extend to $B(\mathfrak{S}_n; \gamma)$. 

In the following proposition we are able to establish a lower bound for the number of extendable $p'$-character degrees lying in any given $p$-block of $\mathfrak{A}_n$. We believe this statement of independent interest from the topic of this article. 

\begin{Prop}\label{cor: an}
Let $n=\sum_{i=0}^ka_ip^{i}$ be the $p$-adic expansion of $n$, with $a_k\neq 0$. Let $\gamma\vdash a_0$, then \[|\mathrm{cd}_{p'}^{ext}(B(n; \gamma))|\geq \lfloor \frac{a_{k}+1}{2}\rfloor\cdot\prod_{i=1}^{k-1}(a_i+1).\]
\end{Prop}
\begin{proof}
By definition, for every partition $\lambda\in\Omega(n;\gamma)$ we have that $(\chi^\lambda)_{\mathfrak{A}_n}$ is a $p'$-degree character that lies in $B(n; \gamma)$ and extends to $\chi^\lambda$ in $B(\mathfrak{S}_n; \gamma)$. The statement now follows from \autoref{lem:alt3}.
\end{proof}

\begin{Prop}\label{prop:alt}
Let $n\geq 5$ be a natural number and $p>3$ be a prime. Then \autoref{simple} holds for $\mathfrak{A}_n$. 
In particular if $n\geq 7$ then $|\mathrm{cd}_{p'}^{ext}(B_0(\mathfrak{A}_n))|\geq 3$.
\end{Prop}
\begin{proof}
Direct verification proves that \autoref{simple} holds for $\mathfrak{A}_5$ and $\mathfrak{A}_6$. 
Suppose that $n\geq 7$ and that $n=a_0+\sum_{i=1}^ka_ip^{n_i}$ is the $p$-adic expansion of $n$, with $a_i\neq 0$ for all $i\geq 1$ and with $n_1<n_2<\cdots<n_k$. Since $p$ is odd, for $P\in\mathrm{Syl}_p(\mathfrak{S}_n)$ we have that $P\leq \mathfrak{A}_n$ and hence that all irreducible characters in $B_0(\mathfrak{A}_n)$ are $P$-invariant. Thus we just need to show that $|\mathrm{cd}_{p'}^{ext}(B_0(\mathfrak{A}_n))|\geq 3$.
From \autoref{cor: an}, we deduce that $|\mathrm{cd}_{p'}^{ext}(B_0(\mathfrak{A}_n))|\geq 3$, whenever $k\geq 3$.   
Suppose that $k\leq 2$. 
If $a_0\leq 1$ then
$\mathrm{Irr}_{p'}(B_0(\mathfrak{A}_n))=\mathrm{Irr}_{p'}(\mathfrak{A}_n)$ and 
 the statement follows from \cite[Proposition 3.5]{GRS}. Hence we can assume that $a_0\geq 2$ and consider $\lambda, \mu\in \mathcal{P}(n)$ to be defined as follows. 
\begin{align*}
\lambda=(a_0, 1^{n-a_0}),&& \text{and}&& \mu=(a_0, 2,1^{n-a_0-2}).
\end{align*}
It is clear that both $(\chi^\lambda)_{\mathfrak{A}_n}$ and $(\chi^\mu)_{\mathfrak{A}_n}$ lie in the principal $p$-block of $\mathfrak{A}_n$ and extend to the principal $p$-block of $\mathfrak{S}_n$, to $\chi^\lambda$ and $\chi^\mu$ respectively. Moreover $\lambda$ and $\mu$ label characters of degree coprime to $p$ by \autoref{lem: last layer}. 
Using the hook-length formula we verify that $1=\chi^{(n)}(1)<\chi^\lambda(1)<\chi^\mu(1).$ The proof is complete. 
\end{proof}

\section{Sporadic groups and groups of Lie type}

\begin{Prop}\label{prop:sporadic}
\autoref{simple} holds for all sporadic simple groups $S$ and the Tits group $\tw{2}F_4(2)'$.
\end{Prop}
\begin{proof}
Recall that $|\Aut(S):S|\le 2$. Hence, we may take a $p'$-character $\hat\alpha$ in the principal block of $\Aut(S)$ such that $\alpha:=\hat\alpha_S\ne 1$ is irreducible. For $\beta$ we can choose any non-trivial $p'$-character in the principal block of $S$. 
Now it can be checked with GAP~\cite{GAP48} that there are choices such that $\beta(1)\nmid\alpha(1)$.
\end{proof}

Now we consider simple groups $S$ of Lie type, by which we mean groups of the form $G/\Z(G)$, where $G=\bg{G}^F$ is the set of fixed points of a simple simply connected algebraic group under a Steinberg morphism $F$.  In the case where $\Z(G)$ is trivial, we define $\wt{{G}}={G}$, and otherwise we let $\bg{G}\hookrightarrow\wt{\bg{G}}$ be a regular embedding, as in \cite[Section 15]{CE04}, so that $\Z(\wt{\bg{G}})$ is connected, $[\wt{\bg{G}}, \wt{\bg{G}}]=[\bg{G}, \bg{G}]$, and $G$ is normal in $\wt{G}:=\wt{\bg{G}}^F$. We write $\wt{S}$ for the group $\wt{G}/\Z(\wt{G})$, so $\Aut(S)$ may be viewed as generated by $\wt{S}$ and graph and field automorphisms. 

Recall that the set $\irr{\wt{G}}$ can be partitioned into so-called Lusztig series $\mathcal{E}(\wt{G},s)$, where $s$ is a semisimple element of the dual group $\wt{G}^\ast$, up to conjugacy. Each series $\mathcal{E}(\wt{G},s)$ has a unique character of degree $[\wt{G}^\ast: C_{\wt{G}^\ast}(s)]_{q'}$, where $\F_{q}$ is the field over which $G$ is defined, called a semisimple character.
Further, the characters in the series $\mathcal{E}(\wt{G},1)$ are called unipotent characters, and for a prime $p$, any $p$-block containing a unipotent character is called a unipotent block.  
 
When $\bg{G}$ is type $A_{n-1}$ (that is, in the case of linear and unitary groups), we will use the notation $PSL_n^\epsilon(q)$ to denote $PSL_n(q)$ for $\epsilon=1$ and $PSU_n(q)$ for $\epsilon=-1$, and similar for $GL_n^\epsilon(q)$ and $SL_n^\epsilon(q)$.  Similarly, $A_{n-1}^\epsilon(q)$ will denote the untwisted case $A_{n-1}(q)$ when $\epsilon=1$ and the twisted case $\tw{2}A_{n-1}(q)$ when $\epsilon=-1$.  We also remark that the group $P\Omega^+_{2n}(q)$ corresponds to $D_n(q)$ and $P\Omega^-_{2n}(q)$ corresponds to $\tw{2}D_n(q)$. 

The following result settles \autoref{simple} for most simple groups in defining characteristic.

\begin{Prop}\label{prop:defininggen}
Let $S$ be a simple group of Lie type defined over $\F_q$, where $q$ is a power of $p>3$ not in the following list: 
$PSL_2(q)$, $PSL_3^\epsilon(q)$, or $PSp_4(q)$.  Then there exist two non-trivial characters $\chi_1, \chi_2\in\irrp {B_0(S)}$ such that $\chi_1(1)\neq\chi_2(1)$ and:
\begin{itemize}
\item If $S\neq P\Omega_8^+(q)$, then for every $S\leq T\leq \Aut(S)$, each of $\chi_1$ and $\chi_2$ extend to a character in the principal $p$-block of $T$.
\item If $S=P\Omega_8^+(q)$, then $\chi_1(1)>2\chi_2(1)$ and for every $S\leq T\leq \Aut(S)$, for $i=1, 2$, there exist $\wh{\chi}_i$  in the principal $p$-block of $T$ such that $\wh{\chi}_i|_S\in\{\chi_i, 2\chi_i\}$. 
\end{itemize}
\end{Prop}
\begin{proof}
In the proof of \cite[Proposition 4.3]{GRS}, it is shown that there exist two characters $\chi_1, \chi_2\in\irrp {\wt{S}}$ that restrict irreducibly to $S$, extend to characters of $\Aut(S)$, have different degrees, and are obtained from characters of $\wt{G}$ trivial on $\Z(\wt{G})$.  Now, since $\irrp {\wt{G}}=\irrp{B_0(\wt{G})}$ (using, for example, \cite[6.18, 6.14, 6.15]{CE04}) and using \cite[Lemma 17.2]{CE04}, we see that in fact these characters are members of the principal block of $\wt{S}$, and their restrictions are members of the principal block of $S$.  

Now, let $S\leq T\leq \Aut(S)$.  Then for $i=1,2$, $\chi_i|_{T\cap \wt{S}}$ is in the principal block of $T\cap \wt{S}$, since $B_0(\wt{S})$ covers a unique block of $T\cap\wt{S}$. Note that by \cite[Theorem 9.4]{Navarro}, there must be a character of $B_0(T)$ lying above $\chi_i|_{T\cap\wt{S}}$.  If $S\neq P\Omega_8^+(q)$, we have $\Aut(S)/\wt{S}$ is abelian, and hence every character of $T$ lying above $\chi_i|_{T\cap\wt{S}}$ is an extension, completing the proof in this case.

If $S=P\Omega_8^+(q)$, then $\Aut(S)/\wt{S}$ is of the form $\mathfrak{S}_3\times C$, where $C$ is cyclic. Then the character $\wh{\chi}_i$ in $B_0(T)$ lying above $\chi_i|_{T\cap\wt{S}}$ must be such that $\wh{\chi}_i|_S\in\{\chi_i, 2\chi_i\}$, as desired.  Switching the roles of the semisimple elements $s_1$ and $s_2$ constructed in \cite[Proposition 4.3]{GRS}, we further see that the characters have been constructed to satisfy $\chi_1(1)>2\chi_2(1)$, since the centralizers of $s_1$ and $s_2$ have types $A_1 \times T_1$ and $A_1^3 \times T_2$ with $T_1$ and $T_2$ appropriate tori, and $2|C_{{G}^\ast}(s_1)|_{p'}<|C_{{G}^\ast}(s_2)|_{p'}$.  \end{proof}

The following handles the exceptional cases left by \autoref{prop:defininggen}.

\begin{Prop}\label{prop:definingexceptions}
Let $S$ be one of $PSL_2(q)$, $PSL_3^\epsilon(q)$, or $PSp_4(q)$, where $q$ is a power of a prime $p>3$.  Then there exist two non-trivial characters $\chi_1, \chi_2\in\irrp {B_0(S)}$ such that $\chi_2(1)\nmid\chi_1(1)$; $\chi_2$ is invariant under a Sylow $p$-subgroup of $\Aut(S)$; and for every $S\leq T\leq \Aut(S)$, $\chi_1$ extends to a character in the principal $p$-block of $T$.
\end{Prop}
\begin{proof}
In this case, characters $\chi_1$ and $\chi_2$ are constructed in the proof of \cite[Lemma 4.4]{GRS} that satisfy all of the needed properties, except possibly the property that for every $S\leq T\leq \Aut(S)$, $\chi_1$ extends to a character in the principal block of $T$.  However, $\chi_1$ is again constructed from a character of $\wt{G}$ trivial on $\Z(\wt{G})$ that restricts irreducibly to $G$.  Hence since again $\Aut(S)/\wt{S}$ is abelian, the proof is complete arguing as in the second paragraph of \autoref{prop:defininggen}.
\end{proof}

For the remainder of the section, we consider the case of non-defining characteristic.  That is, we assume $p>3$ is a prime and $S$ is a simple group of Lie type defined in characteristic different than $p$.  

\begin{Prop}\label{prop:crossgen}
Let $p>3$ be a prime and let $S$ be a simple group of Lie type defined over $\F_q$, where $q$ is a power of a prime different than $p$ and $S$ is not in the following list: 
$PSL_2(q)$, $PSL_3^\epsilon(q)$ with $p\mid (q+\epsilon)$,  $\tw{2}B_2(2^{2a+1})$ with $p\mid (2^{2a+1}-1)$, or $\tw{2}G_2(3^{2a+1})$ with $p\mid (3^{2a+1}-1)$.  Then there exist two non-trivial characters $\chi_1, \chi_2\in\irrp {B_0(S)}$ such that $\chi_1(1)\neq\chi_2(1)$ and:
\begin{itemize}
\item If $S\neq P\Omega_8^+(q)$, then for every $S\leq T\leq \Aut(S)$, each of $\chi_1$ and $\chi_2$ extend to a character in the principal $p$-block of $T$.
\item If $S=P\Omega_8^+(q)$, then $\chi_1(1)>2\chi_2(1)$ and for every $S\leq T\leq \Aut(S)$, for $i=1, 2$, there exist $\wh{\chi}_i$  in the principal $p$-block of $T$ such that $\wh{\chi}_i|_S\in\{\chi_i, 2\chi_i\}$. 
\end{itemize}
\end{Prop}
\begin{proof}
We adapt our proof of \cite[Proposition 4.5]{GRS}, ensuring that we may choose unipotent characters of $p'$-degree satisfying the principal block conditions required here.  That is, we will exhibit unipotent characters of $\wt{G}$ with different degree (and in the case of $P\Omega_8^+(q)$, satisfying $\chi_1(1)>2\chi_2(1)$) that are contained in $\irrp{B_0(\wt{G})}$, which as unipotent characters must be trivial on $\Z(\wt{G})$ and restrict irreducibly to $G$.  Then the restriction lies in $B_0(G)$, since $B_0(\wt{G})$ covers a unique block of $G$, and by \cite[Lemma 17.2]{CE04}, the resulting characters of $\wt{S}$ and $S=G/\Z(G)$ also lie in the principal blocks.  By \cite[Theorems 2.4 and 2.5]{malle08}, every unipotent character extends to its inertia group in $\Aut(S)$, and except for some specifically stated exceptions, the inertia group is all of $\Aut(S)$.  Then arguing as in \autoref{prop:defininggen}, the required properties will hold for each $S\leq T\leq\Aut(S)$.  

To see that the unipotent characters exhibited are indeed of $p'$-degree, it will often be useful to recall that $q^s-1=\prod_{m\mid s}\Phi_m$ and note that $p\mid \Phi_m$ if and only if $m=dp^i$ for some non-negative integer $i$, where $\Phi_m$ denotes the $m$-th cyclotomic polynomial in $q$ and $d$ is the order of $q$ modulo $p$.  Further, $p^2$ divides $\Phi_m$ only if $m=d$. (This is \cite[Lemma 5.2]{malle07}.) 

First, we consider groups of exceptional type.  If $S$ is one of $\tw{2}G_2(3^{2a+1})$ or $\tw{2}B_2(2^{2a+1})$ but not one of the exceptions of the statement, then the unipotent characters mentioned in the proof of \cite[Proposition 4.5]{GRS} work here, since by \cite[Proposition 3.2]{HissbrauerRee}, respectively  \cite[Section 2]{burkhardtSz}, there is a unique unipotent block of maximal defect.  If $S$ is $\tw{2}F_4(2^{2a+1})$, then by \cite[Bemerkung 1]{Malle2F4}, there is again a unique unipotent block of maximal defect unless $p\mid (2^{2a+1}-1)$,  in which case the principal block contains the Steinberg character and two more unipotent characters of $p'$-degree.  Hence we are also done in this case.  If $S=\tw{3}D_4(q)$, then there is either a unique unipotent block of maximal defect or the principal block contains the Steinberg character and one other unipotent character of $p'$-degree, using \cite[Propositions 5.6 and 5.8]{deriziotismichler87}, so we are similarly finished in this case.  

Now let $S$ be one of $G_2(q), F_4(q), E_6(q), \tw{2}E_6(q), E_7(q),$ or $E_8(q)$.  Let $d$ be the order of $q$ modulo $p$.  Using \cite[Theorem A]{enguehard00}, we have the unipotent blocks of $\wt{G}$ are indexed by conjugacy classes of pairs  $(L, \lambda)$ for $L$ a $d$-split Levi subgroup and $\lambda$ a $d$-cuspidal unipotent character.  In particular, the characters in the $d$-Harish-Chandra series indexed by such an $(L, \lambda)$ all lie in the same block of $\wt{G}$.  Further, \cite[Corollary 6.6]{malle07} then yields that if a unipotent character in the series indexed by $(L, \lambda)$ has $p'$-degree, then $L$ is the centralizer of a Sylow $d$-torus.  Now, using this and \cite[Theorem 5.1]{BMM}, we see that either such an $L$ is a maximal torus (yielding a unique block containing unipotent characters of $p'$ degree, and hence we are done using \cite[Proposition 4.5]{GRS} again) or we may use the decompositions in \cite[Table 2]{BMM} to see there are at least two non-trivial unipotent characters in the principal block with different degrees relatively prime to $p$.  (For an example of the argument in the latter situation, consider $E_8(q)$ in the case $d=7$.  Then Line 58 of \cite[Table 2]{BMM} shows that the trivial character and the unipotent characters $\ph_{8, 91}$ and $\ph_{400, 7}$  in the notation of \cite[Section 13.9]{carter}, which have degree $q^{91}\Phi_4^2\Phi_8\Phi_{12}\Phi_{20}\Phi_{24}$ and $\frac{1}{2}q^6\Phi_2^4\Phi_5^2\Phi_6^2\Phi_8\Phi_{10}^2\Phi_{14}\Phi_{15}\Phi_{18}\Phi_{20}\Phi_{24}\Phi_{30}$, respectively, lie in the same $d$-Harish-Chandra series, and hence the same block.  Since $p\mid \Phi_7$ and $p\neq 2$, we see these two non-trivial character degrees are $p'$ and distinct.)

%\smallskip

We are left to consider the classical groups, in which case the unipotent characters of $\wt{G}$ are parametrized by certain partitions or symbols.  By a symbol of rank $n$, we mean a pair of partitions $ {\lambda_1\hbox{ } \lambda_2\hbox{ } \cdots\hbox{ } \lambda_a} \choose {\mu_1\hbox{ } \mu_2\hbox{ } \cdots\hbox{ } \mu_b}$ $ = {\lambda\choose\mu}$, where $\lambda_1<\lambda_2<\cdots<\lambda_a$, $\mu_1<\mu_2<\cdots<\mu_b$, $\lambda_1$ and $\mu_1$ are not both $0$, and $n=\sum_i \lambda_i+\sum_j \mu_j - \lfloor{\left(\frac{a+b-1}{2}\right)^2}\rfloor$.  (The symbol $\lambda\choose\mu$ is equivalent to $\mu\choose\lambda$, and if $\lambda_1$ and $\mu_1$ are both $0$, the symbol is equivalent to $ \lambda_2-1\hbox{ } \cdots\hbox{ } \lambda_a-1 \choose  \mu_2-1\hbox{ } \cdots\hbox{ } \mu_b-1$.)  The defect of a symbol is $|b-a|$. Given an integer $e$, an $e$-hook is a pair of non-negative integers $(x, y)$ with $y-x=e$, $x\not\in\lambda$ (resp. $\mu$), and $y\in \lambda$ (resp. $\mu$).  The $e$-core of a symbol is obtained by successively removing $e$-hooks, which means replacing $y$ by $x$ in $\lambda$ (resp. $\mu$) and then replacing the result with an equivalent symbol satisfying that $\lambda_1$ and $\mu_1$ are not both $0$. An $e$-cohook is defined similarly, except that $x\not\in\lambda$ and $y\in\mu$ (or $x\not\in\mu$ and $y\in\lambda$), and the $e$-cocore is obtained by removing $e$-cohooks, which means removing $y$ from $\mu$ and adding $x$ to $\lambda$ (resp. removing $y$ from $\lambda$ and adding $x$ to $\mu$), and again replacing the result with an equivalent symbol satisfying that $\lambda_1$ and $\mu_1$ are not both $0$.

Tables \ref{unipsA} through \ref{unips2D} describe two unipotent characters for each classical type satisfying the properties described in the first paragraph and not in the list of exceptions of \cite[Theorem 2.5]{malle08}.  For each type, we include a brief discussion, but  we remark that a more complete description of the degrees of such characters and the partitions and symbols can be found in \cite[Section 13.8]{carter}, and a more complete discussion of their distribution into blocks may be found in \cite{fongsrinivason, fongsrinivason89}.  We will include the details for type $A_{n-1}$ in this respect, and note that the other types have similar arguments.

\textbf{Types $A_{n-1}$ and $\tw{2}A_{n-1}$}.
Here $\wt{G}=GL_n^\epsilon(q)$.  In this case, let $e$ be the order of $\epsilon q$ modulo $p$.  The unipotent characters are in bijection with partitions of $n$, and two such characters are in the same block if and only if they have the same $e$-core.  In particular, the trivial character is given by the partition $(n)$, which has $e$-core $(r)$, where $0\leq r< e$ is the remainder when $n:=me+r$ is divided by $e$. Table \ref{unipsA} lists the desired unipotent characters in this case when $n\geq 4$.  Indeed, consider the case $\epsilon=1$.  The partitions listed have $e$-core $(r)$, and hence the corresponding characters are in the principal block and it suffices to show that they have $p'$-degree. Since $p\nmid q$, we need only consider the part of the degree relatively prime to $q$, which are listed following \cite[Section 13.8]{carter}.  If $e=1$, then since $p>3$, the character $\chi_1$ in the cases of line 1 or line 2 has $p'$-degree, since $(q^d-1)/(q-1)$ is divisible by $p$ in this case if and only if $d$ is divisible by $p$. Hence, for $\chi_1$, we may assume $e\neq 1$. Consider line 3 of Table \ref{unipsA} in this case.  Since $me+k$ is not divisible by $e$ for $1\leq k<e$, we see $(q^{me+k}-1)$ contains no factors of the form $\Phi_{ep^i}$.  Hence we see $(q^{me+1}-1)\cdots(q^n-1)$ is not divisible by $p$.  Similarly, if $r+1\neq e$, then $(q^{me-r-1}-1)$ is not divisible by $p$. If $r+1=e$, then $(q^{me-e}-1)/(q^e-1)$ is divisible by $p$ only if $p\mid (m-1)$, so that $(q^{me-e}-1)$ has factors of the form $\Phi_{ep^i}$ with $i\geq 1$.  Hence the character listed in line 3 has $p'$-degree, given the stated conditions, and similar for lines 6 and 7.  Line 5 refers to the Steinberg character, which is certainly of $p'$-degree.  So, consider the characters in lines 4 and 8, of degree $\prod_{i=1}^e\frac{q^{n-i}-1}{q^i-1}$, with $p\mid (m-1)$.  If $p$ divides $\prod_{i=1}^e\frac{q^{n-i}-1}{q^i-1}$, then $p\mid (q^{n-r}-1)/(q^e-1) = (q^{me}-1)/(q^e-1)$, and hence $p\mid m$, a contradiction.   The argument is similar in the case $\epsilon=-1$.   

Finally, if $n=3$ and $p\nmid (q+\epsilon)$, then note that $e=1$ or $3$, $r<2$, and the characters listed in Table \ref{unipsA} still satisfy our conditions.  (In this case, the two characters are the Steinberg character and the unipotent character of degree $q(q+\epsilon)$.)

%\smallskip

\textbf{Types $B_{n}$ and $C_n$}. Here the unipotent characters of $\wt{G}$ are in bijection with symbols of rank $n$ and odd defect.  In this case, we let $e$ be the order of $q^2$ modulo $p$.  Then two symbols are in the same block if and only if they have the same $e$-core, respectively $e$-cocore, if $p\mid q^e-1$, respectively $p\mid q^e+1$.   The trivial character is represented by the symbol $n \choose \emptyset$, which has $e$-core and $e$-cocore $r\choose\emptyset$, where $0\leq r< e$ is the remainder when $n:=me+r$ is divided by $e$.  Table \ref{unipsBC} lists the desired unipotent characters in this case, as long as $n\neq 2$ or $q$ is not an odd power of $2$.  When $n=2$ and $q$ is an odd power of $2$, we have $e=1$ or $2$, so we may still take the Steinberg character for $\chi_2$, but the the characters listed for $\chi_1$ are not necessarily fixed by the exceptional graph automorphism (see \cite[Theorem 2.5(c)]{malle08}).  Here we may instead take the character indexed by $0\hbox{ } 2\choose 1$ of degree $(q+1)^2/2$ when $p\mid (q-1)$, and otherwise we use the character of degree $(q-1)^2/2$ indexed by $0\hbox{ } 1\hbox{ } 2\choose \emptyset$.

%\smallskip

\textbf{Type $D_{n}$ and $\tw{2}D_{n}$}.
In this case the unipotent characters of $\wt{G}$ are in bijection with symbols of rank $n$ and defect $0\pmod 4$, respectively $2\pmod 4$ in case $D_n$, respectively $\tw{2}D_n$.  Again, let $e$ be the order of $q^2$ modulo $p$, and let $n=me+r$ where $0\leq r< e$ is the remainder when $n$ is divided by $e$. The block distribution is described the same way as for types $B_n$ and $C_n$.  

For type $D_n(q)$, the trivial character is represented by the symbol $n \choose 0$, which has $e$-core $r\choose 0 $ if $e\nmid n$ and $\emptyset \choose \emptyset$ if $e\mid n$.  It has $e$-cocore $r\choose 0$ if $m$ is even and $e\nmid n$; $0 \hbox{ } r \choose \emptyset$ if $m$ is odd and $e\nmid n$; $\emptyset \choose \emptyset$ if $m$ is even and $e\mid n$; and $e\choose 0$ if $m$ is odd and $e\mid n$.  Table \ref{unipsD} lists the desired unipotent characters as long as $n\geq 5$. (In some cases, more than two characters are listed.)  We remark that if $n=e$, then it must be that $p\mid (q^e-1)$. 
 
For $D_4(q)=P\Omega_8^+(q)$, note that $1\leq e\leq 3$ and that $p\mid (q^2+1)$ when $e=2$.  Then the Steinberg character of degree $q^{12}$, labeled by $0 \hbox{ } 1 \hbox{ } 2\hbox{ } 3 \choose 1 \hbox{ } 2\hbox{ } 3\hbox{ } 4$ may be taken for $\chi_1$.  For $\chi_2$, we take the character labeled by $3\choose 1$, of degree $q(q^2+1)^2$ when $e=1$ or $3$, and $1\hbox{ } 3\choose 0\hbox{ } 2$ of degree $\frac{1}{2}q^3(q+1)^3(q^3+1)$ when $e=2$.  In either case, we have $\chi_1(1)>2\chi_2(1)$.
 
For type $\tw{2}D_n(q)$, the trivial character is represented by the symbol $0\hbox{ }  n \choose \emptyset$, which has $e$-core $0 \hbox{ }  r \choose \emptyset$ when $e\nmid n$ and $0 \hbox{ } e \choose \emptyset$ if $e\mid n$.  The $e$-cocore is $0 \hbox{ }  r \choose \emptyset$ if $e\nmid n$ and $m$ is even, $r\choose 0$ if $e\nmid n$ and $m$ is odd, $e\choose 0$ if $e\mid n$ and $m$ is even, and $\emptyset\choose\emptyset$ if $e\mid n$ and $m$ is odd.  Table \ref{unips2D} lists the desired unipotent characters in this case.
\end{proof}

\begin{table}\footnotesize
\caption{Some unipotent characters in $\irrp  {B_0(S)}$ for type $A^\epsilon_{n-1}(q)$ with $n\geq 4$ and $p\nmid q$}\label{unipsA}
\begin{tabular}{|c|c|c|c|}
\hline
 & Additional condition on& \multirow{2}{*}{Partition} & \multirow{2}{*}{$\chi(1)_{q'}$}\\
 & $n=me+r$, $r<e$ & & \\
\hline \hline
\multirow{4}{*}{$\chi_1$} & $e=1$ and $p\mid (n-1)$ & $(2, n-2)$ & $\frac{(q^n-\epsilon^n)(q^{n-3}-\epsilon^{n-3})}{(q-\epsilon)(q^2-1)}$ \\
\cline{2-4}
& $e=1$ and $p\nmid (n-1)$ & $(1, n-1)$ & $\frac{q^{n-1}-\epsilon^{n-1}}{q-\epsilon}$ \\
\cline{2-4}
 & $1\neq e\neq r+1$ or $p\nmid (m-1)$ & $(r+1, me-1)$ & 
%$\frac{(q^{me+1}-1)(q^{me+2}-1)\cdots(q^n-1)(q^{me-r-1}-1)}{(q-1)(q^2-1)\cdots(q^{r+1}-1)}$ \\
%&&& 
$\frac{(q^{me+1}-\epsilon^{me+1})(q^{me+2}-\epsilon^{me+2})\cdots(q^n-\epsilon^n)(q^{me-r-1}-\epsilon^{me-r-1})}{(q-\epsilon)(q^2-1)\cdots(q^{r+1}-\epsilon^{r+1})}$\\
\cline{2-4}
& $1\neq e=r+1$ and $p\mid (m-1)$& $(1^{r+1}, me-1)$ & $\prod_{i=1}^{e}\frac{(q^{n-i}-\epsilon^{n-i})}{(q^{i}-\epsilon^i)}$ \\

\hline
\multirow{4}{*}{$\chi_2$} & $r<2$ & $(1^n)$ & 1\\
\cline{2-4}
& $r\geq 2$, $e\neq r+1$ or $m\geq 2$, and & & \\
& $e\neq r+2$ or $p\nmid(m-1)$ & {$(1, r+1, me-2)$} &{$\frac{(q^{me+1}-\epsilon^{me+1})(q^{me+2}-\epsilon^{me+2})\cdots(q^n-\epsilon^n)(q^{me-r-2}-\epsilon^{me-r-2})(q^{me-1}-\epsilon^{me-1})}{(q^{r+2}-\epsilon^{r+2})(q-\epsilon)(q-\epsilon)(q^2-1)\cdots(q^{r}-\epsilon^{r})}$} \\
\cline{2-4}
& $r\geq 2$, $m=1$, $e=r+1$ & {$(1, e-1, e-1)$} &{$\frac{(q^{e+2}-\epsilon^{e+2})(q^{e+3}-\epsilon^{e+3})\cdots(q^n-\epsilon^n)}{(q-\epsilon)(q^2-1)\cdots(q^{e-2}-\epsilon^{e-2})}$} \\
\cline{2-4}
& $r\geq 2$, $e= r+2$,  $p\mid(m-1)$& $(1^{r+2}, me-2)$ & $\prod_{i=1}^{e}\frac{(q^{n-i}-\epsilon^{n-i})}{(q^{i}-\epsilon^i)}$\\
%\cline{2-4}
%& $r\geq 2$, $e= r+1$ & $(1, r, me-1)$ & \\
\hline
\end{tabular}
\end{table}

\begin{table}\footnotesize
\caption{Some unipotent characters in $\irrp  {B_0(S)}$ for types $B_n(q)$, $C_n(q)$ with $n\geq 2$, $p\nmid q$, $(n,q)\neq (2, 2^{2a+1})$}\label{unipsBC}
\begin{tabular}{|c|c|c|c|}
\hline
 & Conditions on& \multirow{2}{*}{Symbol} & \multirow{2}{*}{$\chi(1)_{q'}$ (possibly excluding factors of $\frac{1}{2}$)}\\
 & $n=me+r$, $r<e$ & & \\
\hline \hline
  \multirow{3}{*}{$\chi_1$}
  %& $p\mid(q^e-1)$, $e\mid n$ & $0 \hbox{ } 1\choose n$ & $\frac{(q^{n-1}+1)(q^n+1)}{(q+1)}$ \\
 % \cline{2-4}
%  & $p\mid(q^e+1), e\mid n, \hbox{$m$ even}$    & {$1 \hbox{ } n \choose 0$} & {$\frac{(q^{n-1}-1)(q^n+1)}{(q-1)}$}\\
%\cline{2-4}
% & $p\mid(q^e+1), e\mid n, \hbox{$m$ odd}$ & $0 \hbox{ } n \choose 1$& $\frac{(q^{n-1}+1)(q^n-1)}{q-1}$ \\
% \cline{2-4}
& $p\mid(q^e-1)$  & $0 \hbox{ } r+1 \choose me$ & $\frac{(q^{2(me+1)}-1)\cdots (q^{2n}-1) (q^{me-r-1}+1)(q^{me}+1)(q^{r+1}-1)}{(q^2-1)(q^4-1)\cdots(q^{2(r+1)}-1)}$ \\
%&/// collapse with first row?/// & & \\
\cline{2-4}
 & $p\mid (q^e+1), m $ odd & $0\hbox{ } me \choose r+1$ &  $\frac{(q^{2(me+1)}-1)\cdots (q^{2n}-1) (q^{me-r-1}+1)(q^{me}-1)(q^{r+1}+1)}{(q^2-1)(q^4-1)\cdots(q^{2(r+1)}-1)}$ \\
%&/// collapse with third row?/// & & \\
\cline{2-4}
 & $p\mid (q^e+1), m $ even & $r+1 \hbox{ } me \choose 0$ &  $\frac{(q^{2(me+1)}-1)\cdots (q^{2n}-1) (q^{me-r-1}-1)(q^{me}+1)(q^{r+1}+1)}{(q^2-1)(q^4-1)\cdots(q^{2(r+1)}-1)}$ \\
%&/// collapse with second row?/// & & \\

 \hline
 
   \multirow{3}{*}{$\chi_2$}& $e\mid n$ & $0 \hbox{ } 1 \hbox{ } \cdots \hbox{ } n-1 \hbox { } n \choose 1 \hbox{ } \cdots \hbox{ } n-1 \hbox { } n$ & $1$ \\
  \cline{2-4}
   & $p\mid(q^e-1), e\nmid n, e\neq r+1$ or $p\nmid (m-1)$ & $r+1\hbox{ } me \choose 0$ & $\frac{(q^{2(me+1)}-1)\cdots (q^{2n}-1) (q^{me-r-1}-1)(q^{me}+1)(q^{r+1}+1)}{(q^2-1)(q^4-1)\cdots(q^{2(r+1)}-1)}$ \\
 \cline{2-4}
 & $p\mid(q^e-1), e\nmid n, e= r+1, p\mid(m-1)$ & $0 \hbox{ } me \choose e$ & $\frac{(q^{2(me+1)}-1)\cdots (q^{2n}-1) (q^{me-e}+1)(q^{me}-1)(q^{e}+1)}{(q^2-1)(q^4-1)\cdots(q^{2e}-1)}$\\
 \cline{2-4}
 % & $p\mid(q^e-1), e\nmid n, e= r+1, p\mid(m-1)$ & $0 \hbox{ } e \choose me$ & $\frac{(q^{2(me+1)}-1)\cdots (q^{2n}-1) (q^{me-e}+1)(q^{me}+1)(q^{e}-1)}{(q^2-1)(q^4-1)\cdots(q^{2e}-1)}$ \\
 & $p\mid (q^e+1), e\nmid n, m $ odd & $0\hbox{ } 1\hbox{ } me \choose 1\hbox{ } r+2$ &  $\frac{(q^{2(me+1)}-1)\cdots (q^{2n}-1) (q^{me-r-2}+1)(q^{2(me-1)}-1)(q^{me}-1)}{(q^2-1)^2(q^2-1)(q^4-1)\cdots(q^{2r}-1)(q^{r+2}-1)} $ \\
\cline{2-4}
 & $p\mid (q^e+1), e\nmid n, m $ even & $1\hbox{ } r+2\hbox{ } me \choose 0\hbox{ } 1$ &  $\frac{(q^{2(me+1)}-1)\cdots (q^{2n}-1) (q^{me-r-2}-1)(q^{2(me-1)}-1)(q^{me}+1)}{(q^2-1)^2(q^2-1)(q^4-1)\cdots(q^{2r}-1)(q^{r+2}-1)}$ \\
  \hline
\end{tabular}
\end{table}

\begin{table}\footnotesize
\caption{Some unipotent characters in $\irrp  {B_0(S)}$ for type $D_n(q)$ with $n\geq 5$, $p\nmid q$}\label{unipsD}
\begin{tabular}{|c|c|c|c|}
\hline
 & Conditions on& \multirow{2}{*}{Symbol} & \multirow{2}{*}{$\chi(1)_{q'}$ (possibly excluding factors of $\frac{1}{2}$)}\\
 & $n=me+r$, $r<e$ & & \\
\hline \hline
  \multirow{5}{*}{$\chi_1$}&  $e\nmid n$  & {$me \choose r$} & {$\frac{(q^{2(me+1)}-1)\cdots (q^{2(n-1)}-1) (q^n-1)(q^{me-r}+1)}{(q^2-1)(q^4-1)\cdots(q^{2r}-1)}$} \\

  \cline{2-4}
& $p\mid(q^e-1)$, $e\mid n$; or & \multirow{3}{*}{$0\hbox{ } 1 \hbox{ $\cdots$ } n-1  \choose 1\hbox{ $\cdots$ } n$} & \multirow{3}{*}{$1$} \\
& $p\mid(q^e+1)$, $e\mid n$, $m$ even; or &  & \\
& $e\mid (n-1)$ & &  \\
  \cline{2-4}
%& $p\mid(q^e+1)$, $1\neq e\mid n$, $m$ odd, $p\nmid(m-1)$ & $0\hbox{ } 1\hbox{ } e\hbox{ } n-e+1\choose \emptyset $ & $\frac{(q^{2(n-e+2)}-1)\cdots (q^{2(n-1)}-1) (q^{n}-1)(q^{n-2e+1}-1)(q^{n-e+1}-1)(q^{n-e}-1)}{(q+1)(q^2-1)(q^4-1)\cdots(q^{2(e-2)}-1)(q^{e-1}+1)(q^e+1)}$ \\
%\cline{2-4}
& $p\mid(q^e+1)$, $1\neq e\mid n$, $m$ odd & $1\hbox{ } n-e \choose 0 \hbox{ } e+1$   & $\frac{(q^{2(n-e+1)}-1)\cdots (q^{2(n-1)}-1) (q^{n}-1)(q^{n-2e-1}+1)(q^{n-e}+1)(q^{n-e-1}-1)}{(q-1)(q^2-1)(q^4-1)\cdots(q^{2(e-1)}-1)(q^{e}-1)(q^{e+1}+1)}$ \\
  \hline
  \hline
    \multirow{8}{*}{$\chi_2$}& $p\mid(q^e-1)$, $e\nmid n$ & $1\hbox{ } me \choose 0\hbox{ } r+1$ & $\frac{(q^{2(me+1)}-1)\cdots (q^{2(n-1)}-1) (q^n-1)(q^{me-r-1}+1)(q^{me}+1)(q^{me-1}-1)}{(q^2-1)(q^4-1)\cdots(q^{2(r-1)}-1)(q^{r}-1)(q^{r+1}+1)(q-1)}$ \\

 \cline{2-4}
&  $ e\mid n$, $e\neq 1$ or $p\nmid (n-1)$,   & \multirow{2}{*}{$1\hbox{ } n   \choose 0\hbox{  } 1$} &  \multirow{2}{*}{$\frac{(q^{2(n-1)}-1)}{(q^2-1)}$}\\
& with $p\mid(q^e-1)$ or $m$ even && \\
% & $p\mid(q^e+1)$, $ e\mid n$, $m$ even, $e\neq 1$ or $p\nmid (n-1)$ & & \\
  \cline{2-4}
& $p\mid (q^2-1)$, $p\mid(n-1)$ & {$n-1\choose 1$} & {$\frac{(q^n-1)(q^{n-2}+1)}{q^2-1}$}\\

\cline{2-4}
& $p\mid(q^e+1)$, $e\nmid n$, $m$ even & $r+1\hbox{ } me \choose 0\hbox{ } 1$ & $\frac{(q^{2(me+1)}-1)\cdots (q^{2(n-1)}-1) (q^n-1)(q^{me-r-1}-1)(q^{me}+1)(q^{me-1}+1)}{(q^2-1)(q^4-1)\cdots(q^{2(r-1)}-1)(q^{r}-1)(q^{r+1}-1)(q+1)}$ \\
\cline{2-4}
& $p\mid(q^e+1)$, $e\nmid n$, $m$ odd & $0\hbox{ } me \choose 1\hbox{ } r+1$ & $\frac{(q^{2(me+1)}-1)\cdots (q^{2(n-1)}-1) (q^n-1)(q^{me-r-1}+1)(q^{me}-1)(q^{me-1}+1)}{(q^2-1)(q^4-1)\cdots(q^{2(r-1)}-1)(q^{r}+1)(q^{r+1}-1)(q-1)}$ \\
  \cline{2-4}
& $p\mid(q^e+1)$, $e\mid n$, $m$ odd, $p\nmid(m-2)$ & $n-e\choose e$ & $\frac{(q^{2(me-e+1)}-1)\cdots (q^{2(me-1)}-1) (q^{me}-1)(q^{me-2e}+1)}{(q^2-1)(q^4-1)\cdots(q^{2e}-1)}$ \\
\cline{2-4}
& $p\mid(q^e+1)$, $e\mid n$, $m$ odd, $p\nmid(m-1)$ & $1\hbox{ } n-e+1\choose 0\hbox{ } e$ & $\frac{(q^{2(n-e+2)}-1)\cdots (q^{2(n-1)}-1) (q^{n}-1)(q^{n-2e+1}+1)(q^{n-e+1}+1)(q^{n-e}-1)}{(q-1)(q^2-1)(q^4-1)\cdots(q^{2(e-2)}-1)(q^{e-1}-1)(q^e+1)}$ \\
\hline
  \end{tabular}
  \end{table}
  
  \begin{table}\footnotesize
\caption{Some unipotent characters in $\irrp  {B_0(S)}$ for type $\tw{2}D_n(q)$ with $n\geq 4$, $p\nmid q$}\label{unips2D}
\begin{tabular}{|c|c|c|c|}
\hline
 & Conditions on& \multirow{2}{*}{Symbol} & \multirow{2}{*}{$\chi(1)_{q'}$ (possibly excluding factors of $\frac{1}{2}$)}\\
 & $n=me+r$, $r<e$ & & \\
\hline \hline
  \multirow{6}{*}{$\chi_1$}&  $e\nmid n$  & {$r\hbox{ } me \choose \emptyset$} & {$\frac{(q^{2(me+1)}-1)\cdots (q^{2(n-1)}-1) (q^n+1)(q^{me-r}-1)}{(q^2-1)(q^4-1)\cdots(q^{2r}-1)}$} \\
 \cline{2-4}
 & $p\mid(q^e+1)$, $1\neq e\mid n$, $m$ odd  & \multirow{3}{*}{$0\hbox{  }1\hbox{ } n   \choose  1$} &  \multirow{3}{*}{$\frac{(q^{2(n-1)}-1)}{(q^2-1)}$}\\
 & or &&\\
& $p\mid (q^2-1)$, $p\nmid (n-1)$ & & \\
\cline{2-4}
& $p\mid(q^2-1)$, $p\mid(n-1)$ & $1\hbox{ } n-1\choose \emptyset$ & $\frac{(q^{n}+1)(q^{n-2}-1)}{q^2-1}$\\
\cline{2-4}
& $p\mid(q^e+1)$, $1\neq e\mid n$, $m$ even & $0 \hbox{ }1\hbox{ }  n-e \choose e+1 $   & $\frac{(q^{2(n-e+1)}-1)\cdots (q^{2(n-1)}-1) (q^{n}+1)(q^{n-2e-1}+1)(q^{n-e}-1)(q^{n-e-1}-1)}{(q+1)(q^2-1)(q^4-1)\cdots(q^{2(e-1)}-1)(q^{e}-1)(q^{e+1}-1)}$ \\
\cline{2-4}
& $p\mid(q^e-1)$, $1\neq e\mid n$ & $1 \hbox{ }e+1\hbox{ }  n-e \choose 0 $   & $\frac{(q^{2(n-e+1)}-1)\cdots (q^{2(n-1)}-1) (q^{n}+1)(q^{n-2e-1}-1)(q^{n-e}+1)(q^{n-e-1}-1)}{(q-1)(q^2-1)(q^4-1)\cdots(q^{2(e-1)}-1)(q^{e}+1)(q^{e+1}-1)}$ \\
\hline
 \hline
     \multirow{8}{*}{$\chi_2$}& $p\mid(q^e-1)$, $e\nmid n$ & $0\hbox{ } 1\hbox{ } r+1 \choose  me$ & $\frac{(q^{2(me+1)}-1)\cdots (q^{2(n-1)}-1) (q^n+1)(q^{me-r-1}+1)(q^{me}+1)(q^{me-1}+1)}{(q^2-1)(q^4-1)\cdots(q^{2(r-1)}-1)(q^{r}+1)(q^{r+1}+1)(q+1)}$ \\
 \cline{2-4}
 & $p\mid(q^e+1)$, $e\nmid n$, $m$ even & $1\hbox{ } r+1\hbox{ } me \choose  0$ & $\frac{(q^{2(me+1)}-1)\cdots (q^{2(n-1)}-1) (q^n+1)(q^{me-r-1}-1)(q^{me}+1)(q^{me-1}-1)}{(q^2-1)(q^4-1)\cdots(q^{2(r-1)}-1)(q^{r}+1)(q^{r+1}-1)(q-1)}$ \\
 \cline{2-4}
 & $p\mid(q^e+1)$, $e\nmid n$, $m$ odd & $0\hbox{ } 1\hbox{ } me \choose  r+1$ & $\frac{(q^{2(me+1)}-1)\cdots (q^{2(n-1)}-1) (q^n+1)(q^{me-r-1}+1)(q^{me}-1)(q^{me-1}-1)}{(q^2-1)(q^4-1)\cdots(q^{2(r-1)}-1)(q^{r}-1)(q^{r+1}-1)(q+1)}$ \\
 \cline{2-4}
 & $p\mid(q^e+1)$, $e\mid n$, $m$ odd& \multirow{3}{*}{$0\hbox{ } 1 \hbox{ $\cdots$ } n  \choose 1\hbox{ $\cdots$ } n-1$} & \multirow{3}{*}{$1$} \\
 & or & & \\
 & $e\mid (n-1)$ & & \\
 \cline{2-4}
& $p\mid(q^e+1)$, $1\neq e\mid n$, $m$ even & $0 \hbox{ }e+1\hbox{ }  n-e \choose 1 $   & $\frac{(q^{2(n-e+1)}-1)\cdots (q^{2(n-1)}-1) (q^{n}+1)(q^{n-2e-1}-1)(q^{n-e}-1)(q^{n-e-1}+1)}{(q-1)(q^2-1)(q^4-1)\cdots(q^{2(e-1)}-1)(q^{e}-1)(q^{e+1}+1)}$ \\
\cline{2-4}
& $p\mid(q^e-1)$, $1\neq e\mid n$, $p\nmid(m-2)$ &  $e\hbox{ } n-e\choose \emptyset$ & $\frac{(q^{2(n-e+1)}-1)\cdots (q^{2(n-1)}-1) (q^{n}+1)(q^{n-2e}-1)}{(q^2-1)(q^4-1)\cdots(q^{2e}-1)}$ \\
\cline{2-4}
& $p\mid(q^e-1)$, $1\neq e\mid n$, $p\nmid(m-1)$ & $0\hbox{ } 1\hbox{ } n-e+1\choose  e$ & $\frac{(q^{2(n-e+2)}-1)\cdots (q^{2(n-1)}-1) (q^{n}+1)(q^{n-2e+1}+1)(q^{n-e+1}-1)(q^{n-e}-1)}{(q+1)(q^2-1)(q^4-1)\cdots(q^{2(e-2)}-1)(q^{e-1}-1)(q^e-1)}$ \\
\hline
\end{tabular}
\end{table}

\begin{Prop}\label{prop:crossexceptions}
Let $p>3$ be a prime and let $q$ be a power of a prime different than $p$.  Let $S$ be one of $PSL_2(q)$, $PSL_3^\epsilon(q)$ with $p\mid(q+\epsilon)$, $\tw{2}B_2(2^{2a+1})$ with $p\mid (2^{2a+1}-1)$, or $\tw{2}G_2(3^{2a+1})$ with $p\mid (3^{2a+1}-1)$.  Then there exist two non-trivial characters $\chi_1, \chi_2\in\irrp {B_0(S)}$ such that $\chi_2(1)\nmid\chi_1(1)$; $\chi_2$ is invariant under a Sylow $p$-subgroup of $\Aut(S)$; and for every $S\leq T\leq \Aut(S)$, $\chi_1$ extends to a character in the principal $p$-block of $T$.
\end{Prop}
\begin{proof}
First suppose $S$ is $PSL_2(q)$ or $PSL_3^\epsilon(q)$ with $p\mid(q+\epsilon)$.  In these cases the order of $q$ modulo $p$ is $1$ or $2$, and there is a unique unipotent block of maximal defect, so $\chi_1$ may still be taken to be the Steinberg character.   Let $\delta$ be an element of order $p$ in $\F_{q^2}$.  Write $q=\ell^a$, for some prime $\ell\neq p$, and write $a=p^bc$ with $p\nmid c$.  Then $p\mid \ell^{2c}-1$ since the order of $\ell$ modulo $p$ divides $2a$, and hence $2c$.  Then $\delta$ is either fixed or inverted by $F_{\ell}^c$, where $F_\ell$ is the generating field automorphism.  In particular, since the semisimple classes of $\wt{G}^\ast\cong GL_2(q)$, resp. $GL_3^\epsilon(q)$, are determined by their eigenvalues, this means that a semisimple element $s$ of $\wt{G}^\ast$ with eigenvalues $\{\delta, \delta^{-1}\}$, respectively $\{\delta, \delta^{-1}, 1\}$ is conjugate to its image under $F_{\ell}^c$.  Thus the corresponding semisimple character of $ \wt{G}$ is fixed by $F_{\ell}^c$, and hence a Sylow $p$-subgroup of $\Aut(S)$.  Further, $s$ satisfies (1)-(2) of \cite[Section 4.1.1]{GRS}, that is, $s$ is a member of $[\wt{G}^\ast, \wt{G}^\ast]\cong SL_2(q),$ resp. $SL_3^\epsilon(q)$, and is not conjugate to $sz$ for any $z\in \Z(\wt{G}^\ast)$, since $|\delta|\geq 5$.  Then this character is irreducible on $G$ and trivial on the center.  Further, it has degree $(q-\eta)$, where $\eta\in\{\pm1\}$ is such that $p\mid (q+\eta)$ for $PSL_2(q)$, and degree $q^3-\epsilon$ for $PSL_3^\epsilon(q)$ with $p\mid (q+\epsilon)$.  Since $s$ is a $p$-element, the character lies in a unipotent block, and hence $B_0(\wt{G})$, using \cite[Theorem 9.12]{CE04}.  Then as in the first paragraph of \autoref{prop:crossgen}, the corresponding character of $S$ lies in the principal block.  It also has non-trivial degree prime to $q$, which therefore does not divide the degree of the Steinberg character.  Hence this character satisfies our conditions.

Now let $S$ be $\tw{2}B_2(q^2)$ with $q^2=2^{2a+1}$ and  $p\mid(q^2-1)$ and write $2a+1=p^bc$ with $p\nmid c$.  Let $s$ be such that $\gamma^s$ has order $p\mid (2^c-1)$, where $\gamma$ has order $q^2-1$.  Then using \cite[Section 2]{burkhardtSz} and arguing as in the case above, we see that a slight modification of the characters used in \cite[Lemma 4.8]{GRS} works here: we may take $\chi_1$ to be the Steinberg character and $\chi_2$ to be the character $\chi_5(s)$ in CHEVIE notation. 

Finally, let $S$ be $\tw{2}G_2(q^2)$ with $q^2=3^{2a+1}$ and $p\mid (q^2-1)$.  Again write $2a+1=p^bc$ with $p\nmid c$.  Using \cite[Proposition 3.2]{HissbrauerRee}, there is a unique unipotent block of maximal defect, so we may take $\chi_1$ again to be the Steinberg character.  For $\chi_2$, it follows from \cite[Proposition 4.1]{HissbrauerRee} and arguments as above that we may take the character $\chi_{11}(s)$ in CHEVIE notation, where now $s$ is such that $\gamma^s$ has order $p\mid (3^c-1)$ and $\gamma$ has order $q^2-1$.
\end{proof}

\autoref{simple} now follows from Propositions \ref{prop:alt} and \ref{prop:sporadic} through \ref{prop:crossexceptions}, completing the proof of Theorem A.

\section*{Acknowledgments}
The second author is partially supported by the Spanish Ministerio de Ciencia
e Innovaci\'on PID2019-103854GB-I00 and FEDER funds. The third author is supported by the German Research Foundation (\mbox{SA 2864/1-1} and \mbox{SA 2864/3-1}).
The fourth author is partially supported by a grant from the National Science Foundation (Award No. DMS-1801156).  Part of this work was also completed while the second and fourth authors were in residence at the Mathematical Sciences Research Institute in Berkeley, CA, during Summer 2019 under grants from the National Security Agency (Award No.
H98230-19-1-0119), The Lyda Hill Foundation, The McGovern Foundation, and Microsoft Research.  The authors would also like to extend their gratitude to the anonymous referee for their careful reading and comments about the manuscript.

\end{document}